\documentclass[10pt,a4paper]{amsart}
\usepackage[english]{babel}
\usepackage{amssymb,xypic,amscd}
\CompileMatrices

\newtheorem{defn}{Definition}
\newtheorem{thm}[defn]{Theorem}

\newtheorem{cor}[defn]{Corollary}
\newtheorem{lem}[defn]{Lemma}
\newtheorem{prop}[defn]{Proposition}

\theoremstyle{remark}
\newtheorem{rem}[defn]{Remark}
\theoremstyle{remark}
\newtheorem{exam}[defn]{Example}

\numberwithin{equation}{section} \numberwithin{defn}{section}

\newcommand\aut{\operatorname{Aut}}
\newcommand\ed{\operatorname{End}}

\newcommand\gr{\operatorname{Gr}}

\newcommand\gl{\operatorname{Gl}}

\newcommand\End{\operatorname{End}}

\renewcommand\det{\operatorname{Det}}
\newcommand\Det{\operatorname{det}}

\newcommand\tr{\operatorname{tr}}

\newcommand\res{\operatorname{Res}}

\newcommand\id{\operatorname{I}}

\newcommand\limpl[1]{\underset{#1}\varprojlim\,}

\def\mod #1/#2{\kern.06em{\raise1.2pt\hbox{$#1$}}/
      {\raise-1.2pt\hbox{$#2$}}}

\renewcommand\O{{\mathcal O}}

\newcommand\B{{\mathcal B}}

\renewcommand\tilde{\widetilde}

\renewcommand\lim{\limpl{A\in\B}}
\newcommand\beq{
      \setcounter{equation}{\value{defn}}\addtocounter{defn}1
      \begin{equation}}

\begin{document}

\title[Determinants, symbols and reciprocity laws]{Determinants of finite potent endomorphisms, symbols and reciprocity laws}
\author{Daniel Hern\'andez Serrano \\ Fernando Pablos Romo (*)}
\address{Departamento de
Matem\'aticas, Universidad de Salamanca, Plaza de la Merced 1-4,
37008 Salamanca, Espa\~na}
\address{IUFFYM. Instituto Universitario de F\'{\i}sica Fundamental y Matem\'aticas, Universidad de Salamanca, Plaza de la Merced s/n\\ 37008 Salamanca. Spain.}
\email{fpablos@usal.es}\email{dani@usal.es}
\thanks{
\\(*) Corresponding author\\ This work is partially supported by the
DGESYC research contract no. MTM2009-11393, MTM2012-32342 and Castilla y Le\'on
regional government contract SA112A07.}
\maketitle

\begin{abstract} The aim of this paper is to offer an algebraic definition of
infinite determinants of finite potent endomorphisms using linear algebra techniques. It generalizes Grothendieck's determinant for finite rank endomorphisms and is equivalent to the classic analytic definitions.  The theory can be interpreted as a multiplicative analogue to Tate's formalism of abstract residues in terms of traces of finite potent linear operators on infinite-dimensional vector spaces, and allows us to relate Tate's theory to the Segal-Wilson pairing in the context of loop groups.
\end{abstract}

\bigskip

\setcounter{tocdepth}1

\tableofcontents
\bigskip

\section{Introduction}

Infinite determinants have been studied since the second half of the last century. In 1956,  Grothendieck \cite{Gr} developed an algebraic method to compute $\Det (1+u)$, where $u$ is a finite rank endomorphism on an infinite-dimensional vector space. When $u$ is not of  finite rank, there exist different approaches to define $\Det (1+u)$, as far as we know all of them from an analytic point of view. Let us comment on some of them. 

In 1963, in \cite{DS} Dunford and Schwartz defined the infinite determinant $\Det (1+B)$ in terms of the nonzero eigenvalues of a trace class operator $B$ on a separable Hilbert space; also for trace class operators on Hilbert spaces, in \cite{Si} Simon defined (1976) the infinite determinant by: 
$$\Det_{1}(1+B)=1+\sum_{r=1}^{\infty}\tr(\bigwedge^r B)$$
and showed that it satisfies the expected properties of a determinant. In 2001, in \cite{GGK} Gohberg, Goldberg and Krupnik offered a generalization of determinants for trace potent operators on a separable complex Hilbert space. 

In these notes we offer an algebraic construction of infinite determinants of finite potent endomorphims using linear algebra techniques. It generalizes Grothendieck's determinant for finite rank endomorphisms (over infinite dimensional vector spaces). We show that there exists an equivalence between this algebraic definition and the classical analytic definitions commented above. The key point is to use Tate's definition of traces of finite potent endomorphims (\cite{Ta}).

The algebraic construction of the determinant allows us to construct a central extension of groups whose associated cocycle can be interpreted as a multiplicative analogue of Tate's abstract residue. Moreover, this cocycle shows a way to relate the theory of Segal and Wilson for loops groups (\cite{SW}) and Tate's theory of abstract residues (\cite{Ta}). Finally, a reciprocity law for this cocycle is stated in geometric terms and can be thought of as a multiplicative analogue of Tate's theorem of residues. One could say that this self-contained theory could be thought of as an approach to a unified theory of local symbols (in characteristic zero).

The paper is organized as follows. In section \ref{s:pre} we briefly recall Tate's definition of the trace of a finite potent endomorphism and its  properties. 

Section \ref{s:det} is devoted to giving the algebraic construction of the infinite determinant of finite potent endomorphisms, to showing its basic properties and its relationship to the classical definitions in terms of the exterior algebra and the eigenvalues of the endomorphism.

 Section \ref{s:exp} aims to give a definition of an exponential map for finite potent endomorphims (over a field of characteristic zero), and to show the relationship between  the infinite determinant (defined in the last section) and Tate's trace via this exponential. We also include a subsection meant to extend the definition of the infinite determinant of a finite potent endomorphism to the case in which one has an infinite product of finite potent endomorphisms.

Finally, in section \ref{s:central} we prove the existence of a central extension of groups and explicitly compute its associated cocycle. The proof shows that this cocycle is a multiplicative analogue of Tate's abstract residue and offers us a way to see the equivalence between Segal and Wilson's theory of determinants for loop groups and Tate's theory.  The properties of this cocycle (deduced directly from those of Tate's residue) allow us to state a reciprocity law in geometric terms, that is no more than a multiplicative version of the theorem of residues given by Tate.

In \ref{ss:det&Hilb} we show the equivalence between the algebraic construction given here and the classical analytic definitions.

The last appendix, \ref{ss:SW}, is a quick overview of the theory of Segal and Wilson of determinants in the context of loop groups.

\section{Preliminaries}\label{s:pre}


 Let $k$ be an arbitrary field, let $V$ be a $k$-vector space and let $\varphi$ be an endomorphism  of $V$. 
 
 \begin{defn}
 We say that $\varphi$ is \emph{finite potent} if $\varphi^n V$ is finite
dimensional for some $n$. 
\end{defn}
If $\varphi$ is finite potent, a trace $\tr_V(\varphi) \in k$ may
be defined (see \cite{Ta}), having the following properties:
\begin{enumerate}
\item If $V$ is finite dimensional, then $\tr_V(\varphi)$ is the
ordinary trace. \item If $W$ is a subspace of $V$ such that
$\varphi W \subset W$, then: $$\tr_V(\varphi) = \tr_W(\varphi) +
\tr_{V/W}(\varphi)\, .$$ \item If $\varphi$ is nilpotent, then
$\tr_V(\varphi) = 0$. \item If $F$ is a ``finite potent'' subspace
of $\ed (V)$ (i.e., if there exists an $n$ such that for any
family of $n$ elements $\varphi_1,\dots ,\varphi_n \in F$, the
space $\varphi_1\dots \varphi_n V$ is finite dimensional) then
$\tr_V\colon F\longrightarrow k$ is $k$-linear. \item If $f\colon
V'\to V$ and $g\colon V\to V'$ are $k$-linear and $f g$ is finite
potent, then $g f$ is finite potent, and we have: $$\tr_V (f g) = \tr_{V'}
(g f)\, .$$
\end{enumerate}

Properties (1), (2) and (3)
characterize traces, because if $W$ is a finite dimensional
subspace of $V$ such that $\varphi W \subset W$ and $\varphi^n V
\subset W$, for some $n$, then $\tr_V(\varphi) =
\tr_W(\varphi_{\vert_W})$. And, since $\varphi$ is finite potent,
we may take $W = \varphi^n V$ (for details see \cite{Ta}).

\begin{rem}\label{r:AST} M. Argerami, F. Szechtman and R. Tifenbach have
recently shown in \cite{AST} that an endomorphism $\varphi$ is
finite potent if and only if $V$ admits a $\varphi$-invariant
decomposition $V = W_\varphi \oplus U_\varphi$ such that
$\varphi_{\vert_{U_\varphi}}$ is nilpotent, $W_\varphi$ is finite
dimensional, and $\varphi_{\vert_{W_\varphi}} \colon W_\varphi
\overset \sim \longrightarrow W_\varphi$ is an isomorphism. This
decomposition is unique and one again has that $\tr_V(\varphi) =
\tr_W(\varphi_{\vert_{W_\varphi}})$.

In this paper we shall call this decomposition the
$\varphi$-invariant AST-decomposition of $V$.\qed
\end{rem}

Let $M, N$ be two $k$-vector subspaces of $V$.
\begin{defn}\label{d:comm0} $M$ and $N$ are said to be
commensurable if $M + N /M\cap N$ is a finite dimensional vector
space over $k$. We write $M\sim N$ to denote commensurable
subspaces (\cite{Ta}).
\end{defn}
    If we set $M<N$ when $(M+N)/N$ is finite dimensional, or equivalently, if $M\subset N+W$ for some finite dimensional $W$, it is clear that $M\sim N$ if and only if $M<N$ and $N<M$. Commensurability is an equivalent
relation on the set of $k$-vector subspaces of $V$.

Let us fix a vector subspace $V_+ \subset V$  and let us define subspaces $E$, $E_{0}$, $E_{1}$, $E_{2}$ of $\End_{k}(V)$ by:
 \begin{equation}\label{e:E}
 \begin{aligned}
\varphi \in E &\iff \varphi (V_{+})<V_{+}\\
\varphi \in E_{1} &\iff \varphi (V)<V_{+}\\
\varphi \in E_{2} &\iff \varphi (V_{+})<(0)\\
\varphi \in E_{0} &\iff \varphi (V)<V_{+} \, \mbox{ and }\, \varphi(V_{+})<(0)
\end{aligned}
\end{equation}
\begin{prop}\label{p:Tate1}\cite{Ta}[Prop.1] $E$ is a $k$-subalgebra of $\End_{k}(V)$; the $E_{i}$ are two-sided ideals in $E$; the $E$'s depend only on the $\sim$-equivalence class of $V_{+}$; we have $E_{1}\cap E_{2}=E_{0}$ and $E_{1}+E_{2}=E$; $E_{0}$ is finite potent.
\end{prop}
Thus there is a $k$-linear map $\tr_{V}\colon E_{0}\to k$.
\begin{prop}\label{p:Tate2}\cite{Ta}[Prop. 2] Assume either $\varphi \in E_{0}$ and $\psi \in E$, or $\varphi \in E_{1}$ and $\psi \in E_{2}$. Thus, the commutator $[\varphi,\psi]=\varphi \psi - \psi \varphi \in E_{0}$ and has zero trace.
\end{prop}

\section{Infinite determinants of finite potent endomorphisms.}\label{s:det}

    This section is devoted to defining infinite determinants of finite potent
endomorphism over an arbitrary vector space.

\subsection{Algebraic construction of infinite determinants: definition and basic properties.}\quad 
\label{ss:inf-v-s}

Let $k$ be an arbitrary field and let $V$ be a $k$-vector space. Recall that an endomorphism $\varphi$ of $V$ is ``finite potent" if $\varphi^n V$ is finite dimensional for some $n$.  

 Let us consider the $\varphi$-invariant AST-decomposition $V =
    W_\varphi
\oplus U_\varphi$ (see Remark \ref{r:AST}), where $\varphi_{\vert_{U_\varphi}}$ is
nilpotent, $W_\varphi$ is finite dimensional and
$\varphi_{\vert_{W_\varphi}} \colon W_\varphi \longrightarrow
W_\varphi$ is an isomorphism, we define the determinant by: 
$$\Det^k_V(1 +\varphi) := \Det^k_{W_\varphi}(1 + \varphi_{\vert_{W_\varphi}})\,.$$

It satisfies the following properties:
\begin{enumerate}
\item If $V$ is finite dimensional, then $\Det^k_V(1 + \varphi)$
is the ordinary determinant.

\item If $W$ is a subspace of $V$ such that $\varphi W \subset W$,
then: $$\Det^k_V(1 + \varphi) = \Det^k_W(1 + \varphi) \cdot
\Det^k_{V/W}(1 + \varphi)\, .$$

\item If $\varphi$ is nilpotent, then $\Det^k_V(1 + \varphi) = 1$.
\end{enumerate}

Similar to Tate's definition of traces for finite potent
endomorphisms (see \cite{Ta} and also subsection \ref{s:pre}), properties (1), (2) and (3) characterize
determinants, because if $W$ is a finite dimensional subspace of
$V$ such that $\varphi W \subset W$ and $\varphi^n V \subset W$,
for some $n$, then we have:
 $$\Det^k_V(1 + \varphi) = \Det^k_W(1 +\varphi_{\vert_W})\,.$$
\noindent And, since $\varphi$ is finite potent, we may
take $W = \varphi^n V$. It follows from the definition and properties of $\Det^k_V$ that:

\begin{lem} \label{le:sum} If $\varphi \in \ed (V)$ is a finite potent endomorphism and
 $V = V_1\oplus V_2$ is a decomposition of $V$, where $\varphi V_i \subset
 V_i$ for $i\in \{1,2\}$, one has that:
 $$\Det^k_V(1 + \varphi) = \Det^k_{V_1}(1 + \varphi_{\vert_{V_1}}) \cdot \Det^k_{V_2}(1 + \varphi_{\vert_{V_2}})\, .$$
\end{lem}

  Let us recall from \cite{Ta} that a subspace $F$
of $\ed (V)$ is called ``finite potent'' when there exists an $n$
such that for any family of $n$ elements, $\varphi_1,\dots
,\varphi_n \in F$, the space $\varphi_1\dots \varphi_n V$ is
finite dimensional. 

We shall denote $\langle \varphi,\psi \rangle$ as the set of non constant, noncommutative polynomials on $\varphi, \psi$.

\begin{lem} \label{l:propret} 
 If $F$ is a finite potent subspace of $\ed (V)$, then:
$$\Det^k_V(1 + \varphi)\cdot \Det^k_V(1 + \psi) = \Det^k_V[(1 + \varphi) (1 +
\psi)] = \Det^k_V[(1 + \psi)(1 + \varphi)]\, ,$$
\noindent for all endomorphisms $\varphi, \psi  \in F$.
\end{lem}

\begin{proof} 

 Since $\varphi, \psi \in F$ then $\langle \varphi, \psi \rangle$ is a finite potent subspace of $\ed (V)$
such that the finite dimensional space $W = \langle \varphi,\psi \rangle^n V$
satisfies the following properties:
\begin{itemize}
\item $W$ is invariant for $\varphi$ and $\psi$. \item $\varphi^n
V \subseteq W$ and $\psi^n V \subseteq W$. \item $\Det^k_V (1 +
\varphi) = \Det^k_W (1 + \varphi)$ and $\Det^k_V (1 + \psi) =
\Det^k_W (1 + \psi)$.  \item $W$ is invariant for $\varphi + \psi
+ \varphi \psi$; $(\varphi + \psi + \varphi \psi)^n V \subseteq W$
and
$$\Det^k_V [1 + \varphi + \psi + \varphi \psi] = \Det^k_W [1 +
(\varphi + \psi + \varphi \psi)_{\vert_W}]\, .$$
\end{itemize}

Therefore:
$$\begin{aligned} \Det^k_V(1 + \varphi)\cdot
\Det^k_V(1 + \psi) &= \Det^k_W(1 + \varphi_{\vert_W})\cdot
\Det^k_W(1 + \psi_{\vert_W}) = \\ &= \Det^k_W[(1 +
\varphi_{\vert_W})(1 + \psi_{\vert_W})] \\  &= \Det^k_W [1 +
(\varphi + \psi + \varphi \psi)_{\vert_W}] \\ &= \Det^k_V[(1 +
\varphi) (1 + \psi)] = \Det^k_V[(1 + \psi)(1 + \varphi)]\,.
\end{aligned}$$
\end{proof}

\begin{lem} \label{l:werw} Given two finite potent endomorphisms $\varphi, \psi \in \ed
(V)$ such that the commutator $[\varphi,\psi]=\varphi \psi - \psi \varphi$ is of finite rank,
one has that $\varphi + \psi$ is a finite potent endomorphism.
\end{lem}

\begin{proof} Let $N > 0$ be a positive integer such that
$\varphi^N (V) \subseteq W_{\varphi}$ and $\psi^N (V) \subseteq
W_{\psi}$, $W_{\varphi}$ and $W_{\psi}$ being finite dimensional
$k$-vector subspaces of $V$. If $W_{\varphi \psi}$ is a
finite-dimensional subspace of $V$ such that $[\varphi \psi -
\psi \varphi](V) \subseteq W_{\varphi \psi}$, it is clear that
$$(\varphi + \psi)^{2N} (V) \subseteq \sum_{i=1}^{2N-1} \varphi^i
(W_{\psi}) + \sum_{j=1}^{2N-1} \psi^j (W_{\varphi}) + \sum_{h,k =
1}^{2N-1} \varphi^h \psi^k (W_{\varphi \psi})\, ,$$\noindent
hence the statement is deduced.
\end{proof}

    Similarly, one can see that:

\begin{lem}\label{l:finrank} Given two finite potent endomorphisms $\varphi, \psi \in \ed
(V)$ such that $\varphi \psi - \psi \varphi$ is of finite rank,
one has that $\langle \varphi, \psi \rangle$ is a finite potent subspace of $\ed (V)$.
\end{lem}

\begin{cor} \label{cor:werw}  If $\varphi, \psi \in \ed
(V)$ are two commuting finite potent endomorphisms, then
$\langle \varphi, \psi \rangle$ is a finite potent subspace of $\ed (V)$.
\end{cor}

\begin{prop}\label{prop:fp} If $\varphi$ and $\psi$ are finite potent
endomorphisms such that \linebreak $(1 + \varphi)(1 + \psi) = (1+\psi)(1+\varphi)=1$,
then $\langle \varphi, \psi \rangle$ is a finite potent subspace of $\ed (V)$.
\end{prop}

\begin{proof} Since $\varphi \psi = \psi \varphi$, the claim is a
direct consequence of Corollary \ref{cor:werw}.
\end{proof}

\begin{prop} \label{prop:invertible} Let $\varphi$ be a finite potent endomorphism of $V$.
Then: $$\Det^k_V(1 + \varphi) \ne 0 \quad \iff \quad 1 + \varphi \, \text{ is invertible.}$$
\end{prop}

\begin{proof} Assume that $1 + \varphi$ is
invertible and $(1 + \varphi)\delta = \delta (1 + \varphi) = 1$.
Assume too that $\varphi^n V \subseteq W$, with $W$ a finite
dimensional subspace of $V$. Let us denote:
$$\psi := \delta - 1 = -\varphi \delta = -\delta \varphi\,
.$$\noindent 
Then, $\psi^n V = \varphi^n \delta^n V \subseteq
\varphi^n V \subseteq W$ and therefore $\psi$ is a finite potent
endomorphism such that $(1 + \varphi)(1 + \psi) = (1+\psi)(1+\varphi)=1$.

Therefore, it follows from  Proposition
\ref{prop:fp} and Lemma \ref{l:propret} that:
$$\Det^k_V(1 + \varphi) \cdot \Det^k_V(1 + \psi) = \Det^k_V 1 = 1\,
,$$\noindent and, in particular, $\Det^k_V(1 + \varphi) \ne 0$.

    Conversely, if $\Det^k_V(1 + \varphi) \ne 0$ and
$V = W_\varphi \oplus U_\varphi$ is the
 $\varphi$-invariant AST-decomposition,
one has that  $\Det^k_{W_{\varphi}}(1 + \varphi_{\vert_{W_\varphi}})\ne 0$ and
$1 + \varphi_{\vert_{W_\varphi}} \in \ed_k(W_\varphi)$ is
invertible.

    Thus, we can construct $\delta\in \ed_k(V)$ such that $(1
+ \varphi)\delta = \delta (1 + \varphi) = 1$ as follows:
$$\delta (v) = \left \{ {\begin{aligned} (1 +
\varphi_{\vert_{W_\varphi}})^{-1} (v) \quad &\text{ if } \quad
v\in W_\varphi \\ (1 + \sum_{i = 1}^s (-1)^i
\varphi_{\vert_{U_\varphi}}^i) (v) \quad &\text{ if } \quad v\in
U_\varphi\end{aligned}}\right .$$ with
$\varphi_{\vert_{U_\varphi}}^{s+1} = 0$, hence the claim is
proved.
\end{proof}

    Furthermore, when we consider an extension of fields we obtain
the following result:

\begin{prop} If $k\hookrightarrow K$ is a finite extension of
fields, $V$ is a $K$-vector space, and $\varphi \in  \ed_K (V)$
such that $1 + \varphi$ is invertible, then we have:
$$\Det^k_V (1 + \varphi) = N_{K/k} [\Det^K_V (1 + \varphi)]\,
,$$\noindent $N_{K/k}\colon K^\times \longrightarrow k^\times$
being the norm of the extension $k\hookrightarrow K$.
\end{prop}

\begin{proof} If $\Det^K_V (1 + \varphi) = \Det^K_W (1 +
\varphi_{\vert_W})$ for a finite-dimensional $K$-subspace
$W\subset V$, it is clear that $\Det^k_V (1 + \varphi) = \Det^k_W
(1 + \varphi_{\vert_W})$, and the claim therefore follows from the
following well-known property of determinants  on
finite-dimensional vector spaces:
$$\Det_W^k(\phi) = N_{K/k} [\Det^K_W (\phi)]\,,$$\noindent where
$\phi \in \aut_K (W)$.
\end{proof}

\begin{rem} In \cite{Gr} Grothendieck developed an algebraic
theory of determinants {\bf {d{\'e}t (1+u)}} where $u$ is a finite
rank endomorphism on an infinite-dimensional vector space over a
ground field $k$ of characteristic zero. Obviously, each finite
rank endomorphism is a finite potent endomorphism, and it can be verified that $\Det^k_V (1 + u) = \text{d{\'e}t} (1+u)$ for every finite
rank linear operator $u \in \ed (V)$. Hence, our definition of
determinants of finite potent endomorphisms is a generalization of
Grothendieck's algebraic discussion of infinite determinants.
\end{rem}

\subsection{Infinite determinants and the exterior algebra.} \quad \label{ss:Det&Ext}

    For each positive integer $r$, one has that an endomorphism $\varphi \in
\ed_k(V)$ induces an endomorphism $\bigwedge^r \varphi \in
\ed_k(\bigwedge^r V)$, defined as: $$[\bigwedge^r \varphi] (v_1 \wedge
\cdots \wedge v_r) = \varphi (v_1) \wedge \cdots \wedge \varphi
(v_r)\, ,$$\noindent where $\bigwedge^r V$ is the component of
degree $r$ of the exterior algebra of $V$.

    If $\varphi$ is finite potent, then $\bigwedge^r \varphi$ is
also finite potent, and it is clear that:

\begin{lem} \label{l:cote} If $\varphi \in
\ed_k(V)$ is finite potent such that $\varphi^n V \subset W$, then
$\bigwedge^r \varphi$ is nilpotent for all $r > \text{dim}_k (W)$.
\end{lem}

    Moreover, an easy computation on finite dimensional vector
spaces shows that:

\begin{lem} \label{l:fin-end} Let $W$ be a finite dimensional $k$-vector space and
let $\phi \in \ed_k(W)$. One has that:
$$\Det(1 + \phi) = 1 + \sum_{r=1}^{\text{dim } W} \tr [\bigwedge^r
\phi]\, ,$$\noindent where $\tr$ denotes the ordinary trace.
\end{lem}

    Let $\tr_V$ be the trace defined in \cite{Ta} (see also Section \ref{s:pre}) for finite
potent endomorphisms.

\begin{prop} \label{p:deter-char} If $\varphi \in \ed_k(V)$ is finite potent, then:
$$\Det^k_V(1 + \varphi) = 1 + \sum_{r\geq 1} \tr_{\bigwedge^r V} [\bigwedge^r
\varphi]\, .$$
\end{prop}

\begin{proof} Since $\varphi$ is finite potent, if $\varphi^n V \subset W$ for some positive integer $n$, then
$\tr_V(\varphi) = \tr_W(\varphi)$. Similarly, for all $r\geq 1$
one has that $[\bigwedge^r \varphi]^n V \subset \bigwedge^r W$ for the
same $n$, and hence $\tr_{\bigwedge^r V} [\bigwedge^r \varphi] =
\tr_{\bigwedge^r W} [\bigwedge^r \varphi]$.

    Thus, the claim is deduced immediately from the definition of
the determinant $\Det^k_V(1 + \varphi)$ and the statement of Lemma
\ref{l:fin-end}, and the expression makes sense bearing in mind
Lemma \ref{l:cote} (the trace of a nilpotent endomorphism is
zero).
\end{proof}

    Analogously, we have that:

\begin{cor} \label{cor:deter-char} If $\mu \in k$ and $\varphi \in \ed_k(V)$ is finite potent, then:
$$\Det^k_V(1 + \mu \varphi) = 1 + \sum_{r\geq 1} \mu^r \tr_{\bigwedge^r V} [\bigwedge^r
\varphi]\, .$$
\end{cor}

\begin{lem} \label{l:propret2} 
 For each finite potent endomorphism $\varphi$ and for all automorphism $\phi \in \aut_k (V)$, one has
that:
$$\Det^k_V(1 + \phi \varphi \phi^{-1}) = \Det^k_V(1 + \varphi)\, .$$
\end{lem}
\begin{proof}
First, we should recall from \cite{Ta}
-see also Subsection (\ref{s:pre})- that ``if $f\colon
V'\to V$ and $g\colon V\to V'$ are $k$-linear and $f g$ is finite
potent, then $g f$ is finite potent, and $\tr_V (f g) = \tr_{V'}
(g f)$''. Hence for each finite potent endomorphism $\varphi$ and
for all automorphism $\phi$, since $\varphi = \phi^{-1} (\phi
\varphi)$ is finite potent, $\phi \varphi \phi^{-1}$ is also
finite potent and we have:
 $$\tr_V (\phi \varphi \phi^{-1}) = \tr_V(\varphi)\, .$$
Similarly, one has that:
 $$\tr_{\bigwedge^r V} [\bigwedge^r (\phi \varphi \phi^{-1})] = \tr_{\bigwedge^r V} [(\bigwedge^r\phi)
(\bigwedge^r\varphi) (\bigwedge^r \phi^{-1})] = \tr_{\bigwedge^r V}
(\bigwedge^r\varphi)$$
\noindent for all $r>1$, and it follows from
Proposition \ref{p:deter-char} that:
$$\begin{aligned}\Det^k_V(1 + \phi \varphi \phi^{-1}) &=  1 + \sum_{r\geq 1} \tr_{\bigwedge^r V} [\bigwedge^r
(\phi \varphi \phi^{-1})]\\  &= 1 + \sum_{r\geq 1} \tr_{\bigwedge^r
V} [\bigwedge^r \varphi]= \Det^k_V(1 + \varphi)\, .\end{aligned}$$
\end{proof}

A direct consequence of the previous Lemma and of
Proposition \ref{p:deter-char} is as follows:
\begin{cor} If $F\subseteq \ed (V)$ is a finite potent subspace
and $\varphi, \psi \in F$, then $\varphi\psi$ and $\psi\varphi$
are finite potent endomorphisms satisfying the condition that:
$$\sum_{r\geq 1} \tr_{\bigwedge^r V} [\bigwedge^r
(\varphi\psi)] = \sum_{r\geq 1} \tr_{\bigwedge^r V} [\bigwedge^r
(\psi\varphi)] = \sum_{i,j\geq 1} \big (\tr_{\bigwedge^i V}
[\bigwedge^i \varphi]\big )\cdot  \big (\tr_{\bigwedge^j V} [\bigwedge^j
\psi]\big )\, .$$
\end{cor}

    As in the algebraic formalism of determinants of finite-dimensional vector spaces,
to conclude this subsection we shall study the
relationship between the scalar $\Det^k_V (1 + \varphi)$ and the
``infinite wedge product'' of the elements of a basis in an
arbitrary $k$-vector space of infinite countable dimension
$V$.

\smallskip

    Writing $T^\infty V = \bigotimes_{1}^\infty V$
to denote the infinite countable tensor product, we have that
$\bigwedge^\infty V = T^\infty V / J(V)$,
where $J(V)$ is the ideal generated by those tensors such that at least two of their entries are equal. It is clear that $\bigwedge^\infty V$ is also a $k$-vector space.


    Given a basis $\{v_1, \dots , v_n, \dots\}$, one
has that: $$0 \ne v_1 \wedge \dots \wedge v_n
\wedge \dots \in \bigwedge^\infty V\, .$$

    Moreover, every endomorphism $\psi \in \ed (V)$
induces a $k$-linear map defined by:
$$\begin{aligned} \bigwedge^\infty (\psi) \colon \bigwedge^\infty
V & \longrightarrow \bigwedge^\infty V \\ {u}_1
\wedge \dots \wedge {u}_n \wedge \dots &\longmapsto \psi ({u}_1)
\wedge \dots \wedge \psi ({u}_n) \wedge \dots
\end{aligned}$$\noindent for all ${u}_1
\wedge \dots \wedge {u}_n \wedge \dots \in \bigwedge^\infty {\tilde
V}$. For each $\psi, \phi \in \ed (V)$, one has that:
$$\bigwedge^\infty (\psi \circ \phi) = \bigwedge^\infty (\psi) \circ
\bigwedge^\infty (\phi)\, .$$

    Recall now from \cite{LB} that for every  $\phi \in \ed (V)$
possessing an annihilating polynomial of an arbitrary
infinite-dimensional vector space $V$ there exists a Jordan basis
of $V$ associated with $\phi$.

\begin{prop} \label{prop:santa-catalina}  Let $V$ be a $k$-vector
space of infinite countable dimension, let $\phi\in \ed (
V)$ be a nilpotent endomorphism and let $\{w_1, \dots ,
w_n, \dots\}$ be a Jordan basis of $V$
associated with $\phi$. One has that:
$$\bigwedge^\infty (1 + \phi) \big [w_1  \wedge \dots \wedge w_n \wedge
\dots\big ] = w_1  \wedge \dots \wedge w_n
\wedge \dots\, .$$
\end{prop}

\begin{proof} Since $\phi$ is nilpotent, if $\phi^N = 0$, then a Jordan basis of
$V$ for $\phi$ is: $$\bigcup_{i=1}^\infty \{v_i,
\phi (v_i), \dots , \phi^{\mu (i)} (v_i)\}\,
,$$\noindent with $\mu (i) < N$ and $\phi^{\mu (i) + 1} ({v}_i) = 0$ (notice that when running over all $i$, the vectors $v_i, \phi (v_i), \dots , \phi^{\mu (i)} (v_i)$ are precisely the vectors $w_{i}$ of the statement). 

    Bearing in mind that: $$\big ([1 + \phi](v_i)\big ) \wedge \big ([1 + \phi](\phi (v_i))\big )
\wedge \dots \wedge \big ([1 + \phi](\phi^{\mu (i)} ({v}_i))\big ) = v_i \wedge \phi (v_i)\wedge \dots
\wedge \phi^{\mu (i)} (v_i)\,,$$\noindent and using the same argument for all $i$, we deduce the claim.
\end{proof}

    Let us now consider again a finite potent endomorphism
$\varphi$ of an arbitrary $k$-vector space $V$ of infinite
countable dimension with AST-decomposition $V = {W}_\varphi \oplus {U}_\varphi$. Since $\varphi$ admits an
annihilating polynomial, then there exists a Jordan basis $\{{w}_1^\varphi, \dots , w_N^\varphi, {w}_{N+1}^\varphi \dots\}$, where $\{w_1^\varphi, \dots ,
w_N^\varphi\}$ is a Jordan basis of ${W}_\varphi$
for $\varphi_{\vert_{{W}_\varphi}}$, and $\{{w}_{N+1}^\varphi, \dots, w_{N+s}^\varphi, \dots\}$ is a
Jordan basis of the $k$-vector space ${U}_\varphi$ for the
nilpotent endomorphism $\varphi_{\vert_{{U}_\varphi}}$.
Note that ${U}_\varphi$ also has infinite countable
dimension.

    Thus, it follows from the definition of determinant $\Det^k_V (1 +
\varphi)$ and the statement of Proposition
\ref{prop:santa-catalina} the following generalization of a well known result in the finite dimensional case:

\begin{prop} \label{prop:santa} With the previous notation, if $\varphi$ is a finite potent endomorphism
of an arbitrary $k$-vector space $V$ of infinite countable
dimension, one has that: $$\bigwedge^\infty(1 + \varphi) \big
[w_1^\varphi \wedge \dots \wedge w_N^\varphi
\wedge w_{N+1}^\varphi \wedge \dots \big ] = \Det^k_V (1
+ \varphi) \cdot \big [w_1^\varphi \wedge \dots \wedge
w_N^\varphi \wedge w_{N+1}^\varphi \wedge \dots
\big ]\, .$$
\end{prop}
    Accordingly, the relationship between
the scalar $\Det^k_V (1 + \varphi)$ and the ``infinite wedge
product'' of the elements of an arbitrary basis of $V$ is
given by the following:

\begin{thm} \label{th:santa} If $\varphi$ is a finite potent endomorphism
of an arbitrary $k$-vector space $V$ of infinite
countable dimension and $\{v_1, \dots , v_n,
\dots\}$ is an arbitrary basis of $V$, then:
$$\bigwedge^\infty(1 + \varphi) \big
[v_1 \wedge \dots \wedge v_n \wedge \dots] =
\Det^k_V (1 + \varphi) \cdot \big [v_1 \wedge \dots
\wedge v_n \wedge \dots]\, .$$
\end{thm}

\begin{proof} Let $\{{w}_1^\varphi, \dots , w_N^\varphi, {w}_{N+1}^\varphi \dots\}$ be the above Jordan basis of ${
V}$. Considering the isomorphism:
$$\begin{aligned}\phi \colon V &\longrightarrow V \\  v_i
&\longmapsto w_i^\varphi \end{aligned}$$\noindent and
bearing in mind Lemma \ref{l:propret2} and Proposition
\ref{prop:santa}, one has that:
$$\begin{aligned} \bigwedge^\infty(1 + \varphi) &\big
[v_1 \wedge \dots \wedge v_n \wedge \dots] =
\bigwedge^\infty([1 + \varphi]\circ \phi^{-1}) \big [{w}_1^\varphi \wedge \dots \wedge w_n^\varphi \wedge
 \dots ] \\ &= \big
(\bigwedge^\infty( \phi^{-1}) \circ \bigwedge^\infty(1 + [\phi\circ
\varphi\circ \phi^{-1}])\big ) \big [w_1^\varphi \wedge
\dots \wedge w_n^\varphi \wedge \dots ] \\ &= \Det^k_V (1
+ [\phi\circ \varphi\circ \phi^{-1}])\cdot \bigwedge^\infty(
\phi^{-1}) \big [w_1^\varphi \wedge \dots \wedge {\tilde
w}_n^\varphi \wedge  \dots ] \\ &= \Det^k_V (1 + \varphi) \cdot
\big [v_1 \wedge \dots \wedge v_n \wedge \dots]
\,.\end{aligned}$$
\end{proof}

\subsection{Infinite determinant and eigenvalues.}\label{ss:det&eigen}\quad 

Let $\varphi$ be a finite potent endomorphism of a
$k$-vector space $V$ with $AST$-decomposition $V = W_\varphi\oplus
U_\varphi$. Let $\lambda_1, \dots, \lambda_N$ now be the nonzero
eigenvalues of $\varphi$ in the algebraic closure of $k$ (repeated
according their algebraic multiplicities), which coincide with the
eigenvalues of the isomorphism $\varphi_{\vert_{W_\varphi}}$ (of a
finite-dimensional vector space) in the algebraic closure of $k$.

\begin{prop} \label{prop:eigen} If $\lambda_1, \dots, \lambda_N$ are all the nonzero
eigenvalues of $\varphi$, then one has that:
$$\Det^k_V (1 + \varphi) = \prod_{i=1}^N (1 + \lambda_i)\, .$$
\end{prop}

\begin{proof} If $k\hookrightarrow L$ is an extension of the scalar field $k$ that contains all eigenvalues
of $\varphi$, then: $$\Det^k_{W_\varphi} (1 +
\varphi_{\vert_{W_\varphi}}) = \Det^L_{{W_\varphi}\otimes_k L} [1
+ (\varphi_{\vert_{W_\varphi}} \otimes 1)] = \prod_{i=1}^N (1 +
\lambda_i)$$\noindent and the claim is deduced.
\end{proof}

    If $S_r (x_1, \dots ,x_N)$ is the elementary symmetric
function given by: $$S_r(x_1, \dots ,x_N) = \sum_{1\leq i_1 < \dots
< i_r \leq N} x_{i_1}\dots x_{i_r}\, ,$$\noindent a direct
consequence of Proposition \ref{prop:eigen} is:

\begin{cor} \label{cor:eigen} If $\mu \in k$ and $\lambda_1, \dots, \lambda_N$ are all the nonzero
eigenvalues of $\varphi$, then one has that: $$\Det^k_V (1 + \mu
\varphi) = \sum_{i=0}^N \mu^i\cdot S_i(\lambda_1, \dots
,\lambda_N)\, .$$
\end{cor}

\begin{rem} If $\lambda_1, \dots, \lambda_N$ are again all the nonzero
eigenvalues of a finite potent endomorphism $\varphi$, note that
from Corollary \ref{cor:eigen} and Corollary \ref{cor:deter-char}
we can deduce that: $$\tr_{\bigwedge^r V} [\bigwedge^r \varphi] =
S_r(\lambda_1, \dots ,\lambda_N)\, .$$ 
In particular, when $r =1$ we have an algebraic version for finite potent endomorphisms
of Lidskki's Theorem from the equality:
$$\tr_{V} [\varphi] =  \sum_{i=1}^N \lambda_i\, .$$
\end{rem}

\section{Exponential map and infinite determinants.}\label{s:exp} \quad 

\subsection{Exponential map of finite potent endomorphisms.}\quad 

Let $k$ be a field of characteristic zero and let $V$ be a $k$-vector space. Let $k((z))$ be the field of Laurent series and let us denote $V_{z}:=V\otimes_{k}k((z))$ as a $k((z))$-vector space.
\begin{prop}\label{p:expz}
Let $\varphi \in \End_{k}(V)$ be a finite potent endomorphism and let us denote $\varphi \otimes 1$ the induced endomorphism of $V_{z}$. Then, there exists a well-defined exponential map:
\begin{align*}
\exp_{z} \colon \End_{k}(V) & \to \aut_{k((z))}(V_{z})\\
\varphi & \mapsto \exp_{z}(\varphi)=\id \otimes 1+z(\varphi \otimes 1) +\frac{z^2(\varphi \otimes 1)^2}{2}+\frac{z^3(\varphi \otimes 1)^3}{3!}+\cdots 
\end{align*}
Moreover, the endomorphism of $V_{z}$:
$$\bar \varphi=\exp_{z}(\varphi)-\id \otimes 1= z(\varphi \otimes 1) +\frac{z^2(\varphi \otimes 1)^2}{2}+\frac{z^3(\varphi \otimes 1)^3}{3!}+\cdots $$
is finite potent. 
\end{prop}
\begin{proof}
Since $\varphi$ is finite potent there exist a $\varphi$-invariant finite dimensional $k$-subspace $W$ of $V$ such that $\varphi^n(V)=W$. Let $\{w_{1},\dots ,w_{r}\}$ be a basis for $W$. Then, for any $v\in V$ we have:
\begin{align*}
\exp_{z}(\varphi)(v\otimes 1)&=v \otimes 1+z(\varphi \otimes 1)(v\otimes 1)+\cdots +\frac{1}{(n-1)!}(\varphi \otimes 1)^{n-1}(v\otimes 1)+\\
& +s_{1}(z)w_{1}+\cdots +s_{r}(z)w_{r}\,,
\end{align*}
where $s_{i}(z)\in k((z))$. Thus, $\exp_{z}(\varphi)(v\otimes 1)\in V_{z}$ for all $v\in V$, and therefore $\exp_{z}$ is well-defined. Notice that $\exp_{z}(\varphi + \psi)=\exp_{z}(\varphi)\exp_{z}(\psi)$ whenever $\varphi \psi=\psi \varphi$, and thus $\exp_{z}$ is a bijection.

In particular, 
$$\bar \varphi=\exp_{z}(\varphi)-\id \otimes 1= z(\varphi \otimes 1) +\frac{z^2(\varphi \otimes 1)^2}{2}+\frac{z^3(\varphi \otimes 1)^3}{3!}+\cdots $$
is an endomorphism of $V_{z}$ that consists of a finite sum of commuting finite potent endomorphisms (the first $n-1$ terms) plus a finite rank endomorphism of $W_{z}=W\otimes_{k}k((z))$. Accordingly, using Lemma \ref{l:werw} one concludes that $\bar \varphi$ is finite potent.
\end{proof}

\begin{prop}\label{p:detandtr}
If $\varphi \in \End_{k}(V)$ is finite potent, then:
$$\Det_{V_{z}}^{k((z))}\big(\exp_{z}(\varphi)\big)=\exp_{z}(\tr_{V}(\varphi))\in k((z))^\times \,.$$
\end{prop}
\begin{proof}
By Proposition \ref{p:expz} the endomorphism $\bar \varphi=\exp_{z}(\varphi)-1$ is finite potent and therefore:
$$\Det_{V_{z}}^{k((z))}\big(\exp_{z}(\varphi)\big)=\Det_{V_{z}}^{k((z))}\big(1+\bar \varphi\big)$$
is well-defined. Let $W$ be a finite $k$-subspace of $V$ such that $\varphi (W)\subset W$ and $\varphi^n(V)\subset W$, and let us denote $W_{z}=W\otimes_{k}k((z))$. One has that $(\varphi \otimes 1)(W_{z})\subset W_{z}$ and $(\varphi \otimes 1)^n(V_{z})\subset W_{z}$. Since the result is well-known to be true for finite dimensinal vector spaces one finishes:
\begin{align*}
\Det_{V_{z}}^{k((z))}\big(\exp_{z}(\varphi)\big)&=\Det_{W_{z}}^{k((z))}\big(\exp_{z}(\varphi)_{|W_{z}}\big)=\Det_{W_{z}}^{k((z))}\big(\exp_{z}(\varphi_{|W})\big)=\\
&=\exp_{z}\big(\tr_{W}(\varphi_{|W})\big)=\exp_{z}\big(\tr_{V}(\varphi)\big)\,.
\end{align*}
\end{proof}

Consider the subspace $E_{0}$ of $\End_{k}(V)$ of equation (\ref{e:E}). By Proposition \ref{p:Tate1}, the subspace $E_{0}$ depends only on the commensurability class of $V_{+}$.

\begin{prop}\label{p:detprod}
If both $f$ and $g$ lie on $E_{0}$ then:
$$\Det_{V_{z}}^{k((z))}\big(\exp_{z}(f)\cdot \exp_{z}(g)\big)=\Det_{V_{z}}^{k((z))}\big(\exp_{z}(f)\big)\cdot \Det_{V_{z}}^{k((z))}\big(\exp_{z}(g)\big)$$
\end{prop}

\begin{proof}


First of all, let us notice that the product of any two elements in $E_{0}$ is a finite rank endomorphism of $V$. Since $f$, $g \in E_{0}$, in particular they are finite potent endomorphisms and one has that $f^2(V)$ and $g^2(V)$ are finite dimensional subspaces of $V$. Moreover, $\langle f,g \rangle$ is a finite potent subspace of $\End_{k}(V)$ (by Lemma \ref{l:finrank}) and $\langle f,g \rangle^2(V)=W$ is a finite dimensional subspace of $V$ such that:
\begin{itemize}
\item $f^2(V)\subseteq W$ and $f(W)\subseteq W$.
\item $g^2(V)\subseteq W$ and $g(W)\subseteq W$.
\item $fg(V)\subseteq W$ and $fg(W)\subseteq W$.
\item $gf(V)\subseteq W$ and $gf(W)\subseteq W$.
\item $[f,g](V)\subseteq W$ and $[f,g](W)\subseteq W$.
\end{itemize}
Let us write $\exp_{z}(f)=1+\varphi$ and $\exp_{z}(g)=1+\psi$, where:
$$\varphi=zf+\frac{z^2}{2}f^2+\frac{z^3}{3!}f^3+\cdots $$ 
$$\psi=zg+\frac{z^2}{2}g^2+\frac{z^3}{3!}g^3+\cdots $$
are both finite pontent endomorphisms of $V_{z}$ by Proposition \ref{p:expz}. Thus:
$$[\varphi,\psi](V_{z})=[zf,zg](V_{z})+[zf,\frac{z^2}{2}g](V_{z})+[\frac{z^2}{2}f,g](V_{z})+\cdots \subseteq W\otimes_{k}k((z))=W_{z}\,,$$
and therefore $[\varphi,\psi]$ is a finite rank endomorphism of $V_{z}$. Again applying Lemma \ref{l:finrank}, one has that $\langle \varphi,\psi \rangle$ is a finite potent subspace of $\End_{k}(V_{z})$ and using Lemma \ref{l:propret} one concludes:
\begin{align*}
\Det_{V_{z}}^{k((z))}\big(\exp_{z}(f)\cdot \exp_{z}(g)\big)&=\Det_{V_{z}}^{k((z))}\big((1+\varphi)\cdot (1+\psi)\big)=\\
&=\Det_{V_{z}}^{k((z))}\big(1+\varphi\big)\cdot \Det_{V_{z}}^{k((z))}\big(1+\psi\big)=\\
&=\Det_{V_{z}}^{k((z))}\big(\exp_{z}(f)\big)\cdot \Det_{V_{z}}^{k((z))}\big(\exp_{z}(g)\big)\,.
\end{align*}
\end{proof}

\begin{prop}\label{p:detsum}
If both $f$ and $g$ lie on $E_{0}$ then:
$$\Det_{V_{z}}^{k((z))}\big(\exp_{z}(f+g)\big)=\Det_{V_{z}}^{k((z))}\big(\exp_{z}(f)\big)\cdot \Det_{V_{z}}^{k((z))}\big(\exp_{z}(g)\big)\,.$$  
\end{prop}
 \begin{proof}
Since $E_{0}$ is a subspace one has that  $f+g \in E_{0}$ and hence its exponential is well-defined. Therefore, using Proposition \ref{p:detandtr} and bearing in mind that $\tr_{V}$ is $k$-linear on $E_{0}$ (see Proposition \ref{p:Tate1}) one has:
\begin{align*}
\Det_{V_{z}}^{k((z))}\big(\exp_{z}(f+g)\big)&=\exp_{z}\big(\tr_{V}(f+g)\big)=\exp_{z}\big(\tr_{V}(f)+\tr_{V}(g)\big)=\\
&=\exp_{z}\big(\tr_{V}(f)\big)\cdot \exp_{z}\big(\tr_{V}(g)\big)=\\
&=\Det_{V_{z}}^{k((z))}\big(\exp_{z}(f)\big)\cdot \Det_{V_{z}}^{k((z))}\big(\exp_{z}(g)\big)\,.
\end{align*}
\end{proof}

\begin{cor}
If both $f$ and $g$ lie on $E_{0}$ then:
$$\Det_{V_{z}}^{k((z))}\big(\exp_{z}(f+g)\big)=\Det_{V_{z}}^{k((z))}\big(\exp_{z}(f)\cdot \exp_{z}(g)\big)\,.$$
\end{cor}

\subsection{Determinant of an infinite product of exponentials.}\label{ss:detext}\quad

Let us denote the set $A_{n}=\prod_{i=1}^n \exp_{z^i}(E_{0})$ (for $n\geq 1$) and write:
$$a_{n}=\{\exp_{z}(\varphi_{1}),\dots ,\exp_{z^{n-1}}(\varphi_{n-1}),\exp_{z^{n}}(\varphi_{n})\}\,.$$
For each $n'\geq n$, let us consider the maps:
\begin{align*}
\phi_{n'n} \colon  A_{n'}&\to A_{n}\\
a_{n'}& \mapsto a_{n}\,.
\end{align*}
Thus, $\{A_{n},\phi_{n'n}\}$ is an inverse system of sets and $\varprojlim_{n}A_{n}=\prod_{i=1}^{\infty}\exp_{z^{i}}(E_{0})$.

\begin{rem}
Following Proposition \ref{p:expz}, for every $\varphi \in E_{0}$ its exponential $\exp_{z^{i}}(\varphi)$ has the shape:
$$\exp_{z^{i}}(\varphi)=\id \otimes 1+z^{i}(\varphi \otimes 1)\frac{z^2(\varphi \otimes 1)^2}{2}+\frac{z^3(\varphi \otimes 1)^3}{3!}+\cdots \,,$$
where $\bar \varphi=\exp_{z^{i}}(\varphi) - \id \otimes 1$ is a finite potent endomorphism of $V_{z}$. Thus, its determinant is well-defined. \qed
\end{rem}

We shall abuse of notation by using the same symbol $\prod$ to denote both the Cartesian product of sets and the product of elements in a field.

\begin{prop}\label{p:detinfexp}
There exists a well-defined map:
\begin{align*}
\Det_{V_{z}}^{k((z))}\colon A_{n} & \to k((z))\\
a_{n}& \mapsto \Det^{k((z))}_{V_{z}}(a_{n})=\prod_{i=1}^n\Det_{V_{z}}^{k((z))}\big(\exp_{z^{i}}(\varphi_{i})\big)
\end{align*}
\end{prop}

\begin{proof}
This follows from the fact that for each $n\in \mathbb N$ there exists an inclusion:
\begin{align*}
A_{n}= \prod_{i=1}^n \exp_{z^i}(E_{0}) & \hookrightarrow \aut_{k((z))}(V_{z})\\
 a_{n} &\mapsto \prod_{i=1}^n\exp_{z^{i}}(\varphi_{i})
\end{align*}
and hence, since Proposition \ref{p:detprod} shows that:
$$\Det_{V_{z}}^{k((z))}\big(\prod_{i=1}^n\exp_{z^{i}}(\varphi_{i})\big)=\prod_{i=1}^n\Det_{V_{z}}^{k((z))}\big(\exp_{z^{i}}(\varphi_{i})\big)\,,$$
one concludes.
\end{proof}

The aim now is to define the determinant for certain elements in the limit $\prod_{i=1}^{\infty}\exp_{z^{i}}(E_{0})$. Notice that one has:
$$\prod_{i=1}^{\infty} \exp_{z^i}(\varphi_{i})=\id \otimes 1+(\varphi_{1}\otimes 1)z+(\frac{\varphi_{1}\otimes 1}{2}+\varphi_{2}\otimes 1)z^2+\cdots \in \aut_{k((z))}(V_{z})$$
and hence there exists a well-defined map:
\begin{align*}
\varprojlim_{n}A_{n}=\prod_{i=1}^{\infty} \exp_{z^i}(E_{0}) & \hookrightarrow \aut_{k((z))}(V_{z})\\
 \{a_{n}\}=\{\exp_{z}(\varphi_{1}),\dots ,\exp_{z^{n-1}}(\varphi_{n-1}),\exp_{z^{n}}(\varphi_{n}),\dots \} & \mapsto \prod_{i=1}^{\infty}\exp_{z^{i}}(\varphi_{i})\,.
\end{align*}
However, note too that even if we have:
$$\prod_{i=1}^{\infty} \exp_{z^i}(\varphi_{i})=\id \otimes 1+(\varphi_{1}\otimes 1)z+(\frac{\varphi_{1}\otimes 1}{2}+\varphi_{2}\otimes 1)z^2+\cdots \in E_{0}\otimes_{k}k((z))\,,$$
that is, every coefficient of $z^{i}$ is an element in $E_{0}$ (so in particular every coefficient is finite potent), we still don't know whether a common finite dimensional subspace of $V$ exists or not for all coefficients of $z^{i}$. Therefore, we cannot directly define its determinant. 

In order to solve this deficiency, it sufficies to consider elements in the limit $\{a_{n}\} \in \varprojlim_{n} A_{n}$ that are compatible with  $\Det_{V_{z}}^{k((z))}\colon A_{n} \to k((z))$ 
(where $k((z))$ is endowed with the trivial inverse system); that is, elements $\{a_{n}\}\in \varprojlim_{n} A_{n}$ for which there exists $m\in \mathbb N$ such that for each $n\geq m$:
$$\Det_{V_{z}}^{k((z))}(a_{n})=\Det_{V_{z}}^{k((z))}(a_{m})\,.$$

\noindent For these elements it makes sense to give the following definition.
 


\begin{defn}
Let $\{a_{n}\} \in \varprojlim_{n} A_{n}$ be compatible with $\Det_{V_{z}}^{k((z))}\colon A_{n} \to k((z))$. We define its determinant by:
$$\Det_{V_{z}}^{k((z))}(\{a_{n}\}):=\Det_{V_{z}}^{k((z))}(a_{m})\in k((z))\,.$$
\end{defn}
\noindent Therefore, using Proposition \ref{p:detprod} we have:
$$\Det_{V_{z}}^{k((z))}\big(\prod_{i=1}^{\infty}\exp_{z^{i}}(\varphi_{i})\big):=\Det_{V_{z}}^{k((z))}\big(\prod_{i=1}^m\exp_{z^{i}}(\varphi_{i})\big)=\prod_{i=1}^m\Det_{V_{z}}^{k((z))}\big(\exp_{z^{i}}(\varphi_{i})\big)\,.$$

\begin{exam}\label{ex:C_{i}}
Let $\{\varphi_{i}\}_{i\geq 1}$ be a family of elements in $E_{0}$ such that $\tr_{V}(\varphi_{i})=0$ for all $i\geq m$ for some $m\in \mathbb N$. We have that:
\begin{itemize}
\item $\Det_{V_{z}}^{k((z))}\big(\exp_{z^{i+1}}(\varphi_{i})\big)$ is well-defined for all $i\geq 1$ because $\varphi_{i}\in E_{0}$.
\item $\Det_{V_{z}}^{k((z))}\big(\exp_{z^{i+1}}(\varphi_{i})\big)=1$ for all $i\geq m$, since $\Det_{V_{z}}^{k((z))}\big(\exp_{z^{i+1}}(\varphi_{i})\big)=\exp_{z^{i+1}}\big(\tr_{V}(\varphi_{i})\big)$ (see Proposition \ref{p:detandtr}) and $\tr_{V}(\varphi_{i})=0$ for all $i\geq m$.
\end{itemize}
 Then, for all $n \geq m$ we have:
 $$\Det_{V_{z}}^{k((z))}\big(\prod_{i=1}^{n}\exp_{z^{i+1}}(\varphi_{i})\big)=\prod_{i=1}^n\Det_{V_{z}}^{k((z))}\big(\exp_{z^{i+1}}(\varphi_{i})\big)=\prod_{i=1}^{m-1}\Det_{V_{z}}^{k((z))}\big(\exp_{z^{i+1}}(\varphi_{i})\big)$$
 and therefore the $\varphi_{i}$'s are compatible with $\Det_{V_{z}}^{k((z))}$ and we obtain:
 $$\Det_{V_{z}}^{k((z))}\big(\prod_{i=1}^{\infty}\exp_{z^{i+1}}(\varphi_{i})\big)=\prod_{i=1}^{m-1}\Det_{V_{z}}^{k((z))}\big(\exp_{z^{i+1}}(\varphi_{i})\big)\in k((z))$$
\end{exam}

\section{Central extension of groups and Tate's residue.}\label{s:central}\quad 

Using the theory of infinite determinants developed in the previous sections we are about to construct a central extension of groups by giving its associated cocycle explicitly. The importance of this extension lies in the fact that  it can be viewed as the multiplicative analogue of Tate's formalism of abstract residues in terms of traces of finite potent endomorphisms. Finally, a reciprocity law for this coclycle is given, which can be thought of as a multiplicative version of Tate's theorem of residues.

\subsection{Central extension of groups and Tate's residue.}\quad 

Let us recall Tate's definition of the ``residue map'' (\cite{Ta}). ``Let $K$ be a commutative $k$-algebra with unit, $V$ a $K$-module and $V_{+}$ a $k$-subspace of $V$ such that $fV_{+}<V_{+}$ for all $f\in K$. With the notations $E$, $E_{0}$, $E_{1}$, $E_{2}$ of equation (\ref{e:E}) this latter condition means that $K$ operates on $V$ through $E\subset \End_{k}(V)$, and in what follows we shall use the same letter, $f$, to denote an element of $K$ and its image in $E$''.

\begin{thm}(Definition of residue)\cite{Ta}[Thm. 1]\label{t:resandtr}
In the situation just described, there exists a unique $k$-linear ``residue map'':
$$\res_{V_{+}}^{V}\colon \Omega_{K/k}^1 \to k \,,$$
such that for each pair of elements $f$ and $g$ in $K$ we have:
$$\res_{V_{+}}^V(fdg)=\tr_{V}([f_{1},g_{1}])$$
for every pair of endomorphisms $f_{1}$ and $g_{1}$ in $E$ satifying the following conditions:
\begin{enumerate}
\item Both $f\equiv f_{1}$ (mod $E_{2}$) and $g\equiv g_{1}$ (mod $E_{2}$);
\item Either $f_{1}\in E_{1}$ or $g_{1}\in E_{1}$.
\end{enumerate}
\end{thm}
 
``Note that given $f$ and $g$ in $K$ it is always posible to find $f_{1}$ and $g_{1}$ satisfying $(1)$ and $(2)$ because $E=E_{1}+E_{2}$  (see Proposition \ref{p:Tate1}). Thus, $[f_{1},g_{1}]\in E_{1}$ by $(2)$ and $[f_{1},g_{1}]=[f,g]=0$ (mod $E_{2}$) by $(1)$. Hence, by Proposition \ref{p:Tate1} $[f_{1},g_{1}]\in E_{1}\cap E_{2}=E_{0}$ and therefore its trace is well-defined (recall that $E_{0}$ is finite potent)''.

Our next task consists of giving a multiplicative analogue of Tate's residue. In order to obtain this, let us first recall the \emph{Zassenhaus formula}:
\begin{equation}\label{e:Zassenhaus}
\exp_{z} (f_{1}+g_{1})=\exp_{z}(f_{1})\exp_{z}(g_{1})\prod_{i\geq 1}\exp_{z^{i+1}} \big(\frac{-C_{i}(f_{1},g_{1})}{(i+1)!}\big)\,,
\end{equation}

where:
$$C_{1}=[f_{1},g_{1}]\,,\quad C_{2}=2\big[[f_{1},g_{1}],g_{1}\big]-\big[f_{1},[f_{1},g_{1}]\big]\,, $$
$$C_{3}=3\big[[[f_{1},g_{1}],g_{1}],g_{1}\big]-3\big[[f_{1},[f_{1},g_{1}]],g_{1}\big]+\big[f_{1},[f_{1},[f_{1},g_{1}]\big]\,,\quad \dots $$

\begin{rem}\label{r:C_{i}}
It is worth noticing that since $[f_{1},g_{1}]\in E_{0}$ then $\tr_{V}(C_{1})$ is well-defined and $\tr_{V}(C_{i})=0$ for all $i\geq 2$ by Proposition \ref{p:Tate2}.
\end{rem}

 \begin{thm}\label{t:cocycle}
Let $f,g\in K$. The function:
 \begin{align*}
c_{V_{+}}\colon K\times K&\to k((z))^\times \\
(f,g)&\mapsto c_{V_{+}}(f,g):=\Det_{V_{z}}^{k((z))}\big(\exp_{z}(f_{1})\exp_{z}(g_{1})\exp_{z}(-(f_{1}+g_{1}))\big)
\end{align*}
is a $2$-cocycle of $K$ with coefficients in $k((z))^\times$, for every pair endomorphisms $f_{1}$ and $g_{1}$  in $E$ satisfying:
\begin{itemize}
\item $f\equiv f_{1}$ (mod $E_{2}$) and $g\equiv g_{1}$ (mod $E_{2}$);
\item Either $f_{1}\in E_{1}$ or $g_1\in E_{1}$.
\end{itemize}
In particular, there exists a central extension of groups:
$$1\to k((z))^\times \to  \widetilde K_{V_{+}}\to K\to 1\,.$$
\end{thm}

\begin{proof}
Let us first check that $c_{V_{+}}(f,g)$ is well-defined. If we denote:
$$e(f,g)=\exp_{z}(f_{1})\exp_{z}(g_{1})\exp_{z}(-(f_{1}+g_{1}))$$
and make use of the Zassenhaus formula $(\ref{e:Zassenhaus})$, we have:
$$e(f,g)=\prod_{i\geq 1}\exp_{z^{i+1}} \big(\frac{-C_{i}(-g_{1},-f_{1})}{(i+1)!}\big)\,,$$
hence:
$$c_{V_{+}}(f,g)=\Det_{V_{z}}^{k((z))}\big(\prod_{i\geq 1}\exp_{z^{i+1}} \big(\frac{-C_{i}(-g_{1},-f_{1})}{(i+1)!}\big)\big)\,.$$
Using Remak \ref{r:C_{i}} and Example \ref{ex:C_{i}} (with $m=2$) we may conclude that:
$$c_{V_{+}}(f,g)=\Det_{V_{z}}^{k((z))}\big(\prod_{i\geq 1}\exp_{z^{i+1}} \big(\frac{-C_{i}(-g_{1},-f_{1})}{(i+1)!}\big)\big)=\Det_{V_{z}}^{k((z))}\big(\exp_{z^2}(\frac{1}{2}[f_{1},g_{1}])\big)$$
is well-defined since $[f_{1},g_{1}]\in E_{0}$.


It is now clear that $c_{V_{+}}$ can be regarded as a $2$-cochain of the standard complex $C^{\bullet}(K,k((z))^\times)$, and thus it suffices to check the cocycle condition:
$$c_{V_{+}}(f,g)\cdot c_{V_{+}}(f+g,h)=c_{V_{+}}(g,h)\cdot c_{V_{+}}(f,g+h)$$
for all $f,g,h \in K$. However, this follows directly from Proposition \ref{p:detprod} taking into account that:
$$[f_{1},g_{1}]+[f_{1}+g_{1},h_{1}]=[g_{1},h_{1}]+[f_{1},g_{1}+h_{1}]\,,$$
for every endomorphisms $f_{1}$, $g_{1}$ and $h_{1}$ in $E$ satisfying:
\begin{itemize}
\item $f\equiv f_{1}$ (mod $E_{2}$), $g\equiv g_{1}$ (mod $E_{2}$) and $h\equiv h_{1}$ (mod $E_{2}$);
\item At least two of $f_{1}$, $g_{1}$, $h_{1}$ lie on $E_{1}$.
\end{itemize}
\end{proof}

\begin{rem}
If we write $f=\begin{pmatrix} a & b \\ c& d \end{pmatrix}$, with respect to a fixed decomposition of $V=V_{+}\oplus V_{-}$, we can assume  that $f_{1}=\begin{pmatrix} a & 0 \\ 0& 0 \end{pmatrix}$. Therefore, the definition of the cocycle $c_{V_{+}}$ is inspired by Segal and Wilson's cocycle of \cite[Prop.3.6]{SW} (see also equation (\ref{eq:s-g-cocy}) in Appendix B).
\end{rem}

\begin{rem}
Let us note that by Proposition \ref{p:detandtr} and Theorem \ref{t:resandtr}, Theorem \ref{t:cocycle} states that:
{\small$$c_{V_{+}}(f,g)=\Det_{V_{z}}^{k((z))}\big(\exp_{z^2}(\frac{1}{2}[f_{1},g_{1}])\big)=\exp_{z^{2}}\big(\frac{1}{2}\tr_{V}([f_{1},g_{1}])\big)=\exp_{z^{2}}\big(\frac{1}{2}\res_{V_{+}}^{V}(fdg)\big)\,.$$}
Therefore, the cocycle $c_{V_{+}}$ is a multiplicative analogue for the abstract residue defined by Tate in \cite{Ta}. 

\end{rem}

\begin{rem}\label{r:SG}
Notice that the commutator of the central extension:
$$1\to k((z))^\times \to  \widetilde K_{V_{+}}\to K\to 1\,.$$
is given by:
$$\{f,g\}_{V_{+}}^{V}=\exp_{z^{2}}\big(\tr_{V}([f_{1},g_{1}])\big)\,,$$
and therefore has a similar shape to  Segal and Wilson's pairing given in \cite{SW} for loop groups (see \ref{ss:SW}).  
\end{rem}
\begin{rem}
It can be checked that:
$$\Det_{V_{z}}^{k((z))}\big(\exp_{z}(f_{1})\exp_{z}(g_{1})\exp_{z}(-f_{1})\exp_{z}(-g_{1})\big)=\exp_{z^{2}}\big(\tr_{V}([f_{1},g_{1}])\big)\,,$$
which generalizes the well-known formula:
$$\Det \big(\exp(A)\exp(B)\exp(-A)\exp(-B)\big)=\exp\big(\tr[A,B])\big)$$
for bounded operators $A$ and $B$ with trace-class commutator $[A,B]$ (see \ref{ss:det&Hilb} for the analytic approach).
\end{rem}

\subsection{Properties of $c_{V_{+}}$ and reciprocity law.}\quad

Using Tate's properties $(R_{1}), \dots ,(R_{5})$ for the residue $\res_{V_{+}}^{V}$ (see \cite{Ta}), we have the following:
\begin{itemize}
\item [($C_{1}$)] If $A\sim A'$, then $c_{A}(f,g)=c_{A'}(f,g)$. That is, the cocycle depends only on the commensurability class of $V_{+}$. In particular $\widetilde K_{A}=\widetilde K_{A'}$.
\item [($C_{2}$)] If $f(A)\subset A$ and $g(A)\subset A$, then $c_{A}(f,g)=1$. In particular, the central extension $\widetilde K_{A}$ is trivial.
\item [($C_{3}$)] $c_{A}(1,g)=1$ for all $g\in K$.
\item [($C_{4}$)] If $g$ is invertible in $K$ and $h\in K$ is such that $h(A)\subset A$, then:
$$c_{A}(hg^{-1},g)=\exp_{z^{2}}\big(\frac{1}{2}\tr_{A/A\cap gA}(h)\big)\cdot \exp_{z^{2}}\big(-\frac{1}{2}\tr_{gA/A\cap gA}(h)\big)\,.$$
\item [($C_{5}$)] If $B$ is another $k$-subspace of $V$ such that $f(B)<B$ for all $f\in K$, then:
$$c_{A+B}(f,g)\cdot c_{A\cap B}(f,g)=c_{A}(f,g)\cdot c_{B}(f,g)\,.$$
\end{itemize}

Using these properties and the Corollary of \cite[Theorem 3]{Ta} we have:
 \begin{cor}(Reciprocity law)
Let $X$ be a non-singular and complete curve over $k$ and let $K$ denote its function field. For a closed point $x\in X$, let us write $A_{x}=\widehat \O_{X,x}$ and $V_{x}=(\widehat \O_{X,x})_{0}$. Then:
 $$\prod_{x\in C}c_{A_{x}}(f,g)=1 \quad \mbox{ for all } f,g\in K\,,$$
 where the product is taken over all closed points $x$ of $X$.
\end{cor}

 \begin{rem}
 Let us remark that this last section is an approach to a unified theory of local symbols over fields of characteristic zero. Using the algebraic definition of the determinant for finite potent endomorphisms developed along the paper, Theorem \ref{t:cocycle} reveals an equivalence between the theory of Segal and Wilson for loop groups (see \ref{ss:SW}) and the theory of abstract residues given by Tate. 
 
 Further research will be done in order to extend this theory to characteristic $p>0$ and modules over artinian local rings for the purpose of obtaining a unified theory of arithmetic symbols from infinite determinants. 
\end{rem}

\appendix
\section{Analytic approach.}\label{ss:det&Hilb}\quad
 
  Let $H$ be a separable complex Hilbert space. Given an arbitrary finite potent linear operator $\varphi\colon H
\longrightarrow H$, bearing in mind that it has an annihilating
polynomial, it follows again from \cite{LB} that $H$ admits a
Jordan basis for $\varphi$, with arguments similar to those of  M.
Argerami, F. Szechtman and R. Tifenbach in \cite{AST}, one has
that the Hilbert space $H$ admits a $\varphi$-invariant
decomposition $H = W_\varphi \oplus U_\varphi$ such that
$\varphi_{\vert_{U_\varphi}}$ is nilpotent, $W_\varphi$ is finite
dimensional, and $\varphi_{\vert_{W_\varphi}} \colon W_\varphi
\overset \sim \longrightarrow W_\varphi$ is an isomorphism.

    Thus, since all the expressions referred to in the previous
algebraic approach to infinite determinants consist of finite sums
of well-defined analytic elements of $\varphi$, the above
definition and properties of $\Det^{\mathbb C}_H (1 + \varphi) \in
{\mathbb C}$ are valid within the analytic formalism of Hilbert
spaces.

   In fact, in \cite{Si} B. Simon defined determinants of trace
class operators $B$ on a separable Hilbert space from the formula:
\begin{equation} \label{eq:Simon}\Det_1 (1 + \mu B) =1+ \sum_{n=1}^{\infty} \mu^n
\tr(\bigwedge^n(B))\,,\end{equation}
and according to \cite{RS} and \cite{Si}, for
each trace class operator $B$ one has that the infinite determinant
$\Det_1 (1 + B)$ satisfies the following properties:
\begin{itemize}
\smallskip



\item If $A$ and $B$ are trace class, then: $$\Det_1 (1 + A) \cdot
\Det_1 (1 + B) = \Det_1 (1 + A + B + AB) = \Det_1 (1 + B) \cdot
\Det_1 (1 + A)\, .$$

\smallskip

\item The operator $1 + B$ is invertible if and only if $\Det_1 (1
+ B) \ne 0$.

\smallskip

\item If $U$ is unitary, then:
$$\Det_1 (U^{-1} (1 + B) U) = \Det_1 (1 + U^{-1}BU) = \Det_1 (1 + B)\,
.$$
\end{itemize}

Corollary \ref{cor:deter-char} shows that our definition of determinants of
finite potent endomorphisms coincides with expression
(\ref{eq:Simon}) when we replace the usual trace of linear
operators of Hilbert spaces with Tate's definition of traces of
finite potent endomorphisms. The above properties of $\Det_{1}(1+B)$ correspond to Lemma \ref{l:propret}, Proposition \ref{prop:invertible}, and Lemma \ref{l:propret2} (respectively) in the finite potent case. 


    Analogously, let $\{\lambda_i(B)\}_{i=1}^{N(B)}$ be the listing of all nonzero eigenvalues
of a trace class operator $B$, counted up to algebraic
multiplicity. In \cite{DS} Dunford and Schwartz defined the
infinite determinant $\Det_1 (1 + \mu B)$ as the expression:
\begin{equation}\label{eq:eigen}
\Det_1 (1 + \mu B) = \prod_{i=1}^{N(B)} (1 + \mu\lambda_i(B))\,,
\end{equation}
\noindent which coincides with the statement of Proposition
\ref{prop:eigen}. This is another important equivalence
between the classical results of well-known analytic treatments of
infinite determinants and the above algebraic theory.

    Furthermore, if $\lambda_1, \dots, \lambda_N$ are  all the nonzero
eigenvalues of a finite potent and trace class operator $\varphi$,
the algebraic and the analytic versions of the Lidski's Theorem
imply that:
$$\tr_{V} [\varphi] =  \sum_{i=1}^N \lambda_i = \tr [\varphi]\,
,$$\noindent where $\tr_{V} [\varphi]$ is the trace of $\varphi$
as a finite potent endomorphism and $\tr [\varphi]$ is the trace
of $\varphi$ as a linear operator of a Hilbert space.

    Accordingly, from expression (\ref{eq:eigen}) and
Proposition \ref{prop:eigen} we can also deduce that if $H$ is a separable complex Hilbert
space and $\varphi\colon H \longrightarrow H$ is a linear operator
that is finite potent and of trace class, then:
$$\Det^{\mathbb C}_H (1 + \varphi) = \Det_1 (1 + \varphi)\,.$$

    Therefore, the determinant $\Det^k_V(1 + \varphi)$ is an
extension to finite potent endomorphisms on arbitrary vector
spaces of the most usual definition of infinite determinants of
trace class operators on separable Hilbert spaces.

\begin{rem} Let $H$ again be a separable complex Hilbert
space and let $\varphi\colon H \to H$ be a finite potent
endomorphism. If $\lambda_1, \dots, \lambda_N$ are all the nonzero
eigenvalues of $\varphi$, since $\tr_V (\varphi^n) = \lambda_1^n +
\dots + \lambda_N^n = p_n(\lambda_1, \dots, \lambda_N)$, similar
to trace class operators, it follows from the statement of
Corollary \ref{cor:eigen} and from the properties of symmetric
functions (see \cite{Mc}, Chapter 1) that:
\begin{itemize}
\item If $\mu \in {\mathbb C}^\times$ and $\vert \mu \vert \cdot
\text{max} \, \vert \lambda_i\vert < 1$, then
$$\Det^{\mathbb C}_H (1 + \mu \varphi) = \exp (\tr_H [\ln (1 +
\mu \varphi)])\, .$$

    Although the Tate's trace for arbitrary finite
potent endomorphisms is not linear (\cite{PR}), note that $$\tr_H
[\ln (1 + \mu \varphi)] = - \sum_{r=1}^{\infty} \frac{\mu^r}{r}
\tr_H [(-\varphi)^r]\, ,$$\noindent because the
$\varphi$-invariant AST-decomposition of $V$ coincides with the
$\varphi^r$-invariant AST-decomposition of $V$ for all r.

\smallskip

\item  If $\alpha_0 (\varphi) := 1$, $\alpha_1 (\varphi) := \tr_H
(\varphi)$ and
$$\alpha_m (\varphi) = \begin{vmatrix} \tr_H (\varphi) & m-1 & 0 & \dots & 0 & 0 \\  \tr_H(\varphi^2)
& \tr_H (\varphi) & m-2 & \dots& 0 & 0 \\ \tr_H (\varphi^3) &
\tr_H (\varphi^2) & \tr_H (\varphi) & \ddots & 0 & 0\\ \vdots &
\vdots & \vdots & \ddots & \ddots & \vdots \\ \tr_H
(\varphi^{m-1)}) & \tr_H
(\varphi^{m-2)}) & \tr_H (\varphi^{m-3)}) & \dots & \tr_H (\varphi) & 1 \\
\tr_H (\varphi^{m)}) & \tr_H (\varphi^{m-1)}) & \tr_H
(\varphi^{m-2)}) & \dots & \tr_H (\varphi^2) & \tr_H (\varphi)
\end{vmatrix}$$
\noindent for $m \geq 2$, similar to the ``Plemelj-Smithies"
formula, we have that:
\begin{equation} \label{eq:PS-hilb} \Det^{\mathbb C}_H (1 + \mu \varphi) = \sum_{m= 0} ^\infty \mu^m \frac{\alpha_m
(\varphi)}{m!}\, .\end{equation}
\end{itemize}
\end{rem}

\begin{rem}  It is
important to emphasize that there is no relationship of inclusion
between trace class and finite potent endomorphisms, as is deduced
from the following example.

    Let $H$ be a separable complex Hilbert space with an orthonormal
basis $\{e_1,e_2,\dots,$ $e_n,\dots\}$. If we consider the
endomorphisms $\phi_1, \phi_2 \in \ed_{\mathbb C} H$ defined by:
$$\phi_1 (e_i) = \left \{ {\begin{aligned} e_{j+1} \quad
&\text{ if } \quad j \text{ is even } \\ 0 \quad
&\text{ if } \quad j \text{ is odd }\end{aligned}} \right .$$ 
and $\phi_2 (e_s) = \frac{1}{s^{2}}e_{s}$ for all s, then we have that $\phi_1$ is not of trace-class, because $\phi_{1}^{*}\phi$ is an infinite rank projection and, as $\phi_{1}^{2}=0$, it is nilpotent (and thus finite potent). And $\phi_2$ is a trace class operator such that it is not finite potent, since every power of $\phi_{2}$ has dense range, and so it is not finite-rank.
\qed
\end{rem}

\begin{rem} If $\{\tilde{\lambda_i}\}_{i=1}^{N({\tilde A})}$ are the eigenvalues
of an operator ${\tilde A}$ of a separable Hilbert space $H$
(repeated again according to their algebraic multiplicities), the
Carleman-Fredholm determinant is defined as:
$$\Det_2 (1 + {\tilde A}) = \prod_{i= 1}^{N({\tilde A})}(1 +
\tilde{\lambda_i})\exp(-\tilde{\lambda_i})\, ,$$\noindent and this
product is known to converge for  ``Hilbert-Schmidt'' operators.

    Hence, this is a different method from the one described above for defining an
infinite determinant of trace class operators.
\end{rem}

\begin{rem}\label{r:GGK} In 2001, in \cite{GGK}  I. Gohberg, S. Goldberg and N. Krupnik
offered a generalized definition of determinants for
trace-potent operators on a separable complex Hilbert space $H$.
Indeed, a bounded linear operator ${\bar A}$ on $H$ is called
``trace-potent'' if there exists an integer $m > 1$ such that
${\bar A}^m$ is of trace class. Let $m$ be the smallest number for
which ${\bar A}^m$ is of trace class; they defined a m-regularized
determinant by the following equality:
$${\tilde \Det_m} (1 -  \mu{\bar A}) = \prod_{i= 1}^{N({\bar A})}(1
- \mu{\bar {\lambda_i}})E_m(\mu{\bar {\lambda_i}})\, ,$$\noindent
where $\{{\bar {\lambda_i}}\}_{i=1}^{N({\bar A})}$ are the
eigenvalues of ${\bar A}$ (repeated again according to their
algebraic multiplicities), $\mu \in {\mathbb C}$ and: $$E_m
({\delta}) = \exp \big ( \sum_{j=1}^{m-1} \frac{\delta^j}j \big
)\, .$$
    Note that if $\lambda = -1$ and $m = 2$, then ${\tilde \Det_2}$ coincides
with the Carleman-Fredholm determinant $\Det_2$ referred to above.

 Notice that each finite potent endomorphism of a Hilbert
space is also trace-potent (Remark \ref{r:GGK}), and we should
note that our definition of determinant is different from that given
in \cite{GGK}.

\end{rem}

\section{The Segal-Wilson Pairing}\label{ss:SW}

Let $H$ be a separable complex Hilbert space with a given
decomposition $H = H_+\oplus H_-$ as the direct sum of two
infinite dimensional orthogonal closed subspaces.

    The Grassmannian $\gr (H)$ is the set of all closed subspaces
$W$ of $H$ such that
\begin{itemize}
\item the orthogonal projection $\text{pr} \colon W
\longrightarrow H_+$ is a Fredholm operator (i.e. it has finite
dimensional kernel and cokernel), and \item the orthogonal
projection $\text{pr} \colon W \longrightarrow H_-$ is a compact
operator.
\end{itemize}

Let us write the operators $g\in \gl (H)$ in the block form
$$g = \begin{pmatrix} a & b \\ c & d \end{pmatrix}$$\noindent with
respect to the decomposition $H = H_+\oplus H_-$. The
\textit{restricted general linear group} $\gl_{res}(H)$ is the
closed subgroup of $\gl (H)$ consisting of operators $g$ whose
off-diagonal blocks $b$ and $c$ are compact operators. The blocks
$a$ and $d$ are therefore automatically Fredholm.

    Any continuous non-vanishing function $f$ on $S^1$ defines an
invertible multiplication operator, again written $f$, on $H$.
Indeed, if $\Gamma$ denotes the group of continuous maps $S^1
\longrightarrow {\mathbb C}^\times$, regarded as multiplication
operators on $H$, then $\Gamma \subset \gl_{res}(H)$.

    We now consider the subgroup $\Gamma_+$ of $\Gamma$
consisting of all-real analytic functions $f\colon S^1
\longrightarrow {\mathbb C}^\times$ that extend to holomorphic
functions $f\colon D_0 \longrightarrow {\mathbb C}^\times$ in the
disc $D_0 = \{z\in {\mathbb C} : \vert z\vert \leq 1\}$ satisfying
$f(0) = 1$, and the subgroup $\Gamma_-$ of $\Gamma$ consisting of
functions $f$, which extend to non-vanishing holomorphic functions
in $D_{\infty} = \{z\in {\mathbb C} \cup \infty : \vert z\vert
\geq 1\}$ satisfying $f(\infty) = 1$.

    Let us recall that \cite{Si} uses Grothendiek definition of determinant for  an operator of trace class on a separable Hilbert space and study its properties. 

   In \cite{SW} G. Segal and G. Wilson offered the definition of a
holomorphic line bundle $\det$ over $\gr$, as the
bundle whose fiber over $W\in \gr (H)$ is an infinite expression
$\lambda\cdot \omega_0\wedge \omega_1 \wedge \omega_2 \wedge
\dots$ where $\{\omega_1\}$ is what they called an ``admissible
basis'' for $W$. If $\{\omega_i\}$ and $\{\omega'_i\}$ are two
admissible basis of $W$, then the infinite matrix $t$ relating
them is the kind that has a determinant, and it is possible to
assert that:
$$\lambda\cdot \omega_0\wedge \omega_1 \wedge \omega_2 \wedge \dots =
\lambda\cdot  \Det (t) \cdot \omega'_0\wedge \omega'_1 \wedge
\omega'_2 \wedge \dots$$\noindent when $\omega_i = \sum t_{ij}
\omega'_j$.

    Let us now consider the subgroup $\gl_{1}(H)$ of $\gl_{res}(H)$ consisting
of invertible operators $g = \begin{pmatrix} a & b \\ c & d
\end{pmatrix} \in \gl_{res}(H)$ and where the blocks $b$ and $c$ are
of trace class. If $\gl_{1}(H)^0$ is the identity component of
$\gl_{1}(H)$, then the action of $\gl_{1}(H)^0$ on $\gr$ does lift
projectively to $\det$; that is: there exists a central extension
of groups: \begin{equation*}  1 \to {\mathbb
C}^\times \longrightarrow \gl_1^{\wedge} \longrightarrow
\gl_{1}(H)^0 \to 1\, ,\end{equation*}\noindent which acts on
$\det$, covering the action of $\gl_{1}(H)^0$ on $\gr$.

    If we consider the open set $\gl_{1}^{reg}$ of $\gl_{1}(H)^0$
where $a$ is invertible, there exists a section $s\colon
\gl_{1}^{reg} \to \gl_1^{\wedge}$ of the projection
$\gl_1^{\wedge} \longrightarrow \gl_{1}(H)^0$ that induces a
2-cocycle $(\cdot, \cdot) \colon \gl_{1}^{reg} \times
\gl_{1}^{reg} \longrightarrow {\mathbb C}^\times$, defined as:
\begin{equation}\label{eq:s-g-cocy} (g_1,g_2) \longmapsto \Det (a_1 a_2 a_3^{-1})\, ,\end{equation}\noindent where
$g_i =
\begin{pmatrix} a_i & b_i \\ c_i & d_i \end{pmatrix}$ and $g_3 =
g_1 g_2$.

    Accordingly, the elements of  $\gl_{1}^{reg}$ act on $\det$ by means of the section
$s$. However, $\gl_{1}^{reg}$ is not a group, and the map $s$ is
not multiplicative.

    Let us now consider the subgroup  $\gl_{1}^{+}$ of
$\gl_{1}^{reg}$ consisting of elements whose block decomposition
has the form $\begin{pmatrix} a & b \\ 0 & d
\end{pmatrix}$. Thus, the restriction of $s$ to $\gl_{1}^{+}$ is an inclusion of
groups $\gl_{1}^{+} \hookrightarrow \gl_1^{\wedge}$ and one can
regard $\gl_{1}^{+}$ as a group of automorphisms of the bundle
$\det$. Similar remarks apply to the subgroup $\gl_{1}^{-}$
consisting of the elements of $\gl_{1}^{reg}$ whose block
decomposition has the form $\begin{pmatrix} a & 0 \\ c & d
\end{pmatrix}$. In particular, the subgroups $\Gamma_+$ and
$\Gamma_-$ of the group of maps $S^1 \to {\mathbb C}^\times$ act
on $\det$, for $\Gamma_{\pm} \subset \gl_1^{\pm}$.

    Now, for every subgroup $G,{\tilde G} \subset \gl_{1}^{reg}$ such that the action of $G_1$ and $G_2$
commute with each other it is possible to define a map
\begin{equation*} \label{eq:SW-pairing} \begin{aligned}(\cdot,\cdot)_{G,{\tilde G}}^{SW}\colon G
\times {\tilde G} &\longrightarrow {\mathbb C}^\times \\
(g,{\tilde g}) &\longmapsto \Det (a{\tilde a}a^{-1}{\tilde
a}^{-1})\, ,\end{aligned}\end{equation*}\noindent where $g =
\begin{pmatrix} a & b \\ c & d
\end{pmatrix}\in G$ and ${\tilde g} =
\begin{pmatrix} {\tilde a} & {\tilde b} \\ {\tilde c} & {\tilde d}
\end{pmatrix}\in {\tilde G}$. We have that the commutator has a determinant because, from the fact that
$g$ and $\tilde g$ commute, it is equal to $1 - b{\tilde
c}a^{-1}{\tilde a}^{-1} + {\tilde b}ca^{-1}{\tilde a}^{-1}$, and
b, c, ${\tilde b}$ and ${\tilde c}$ are of trace class by the
definition of $\gl_{1}^{reg}$. 

Hence this map is well-defined, and
we shall call it ``the Segal-Wilson pairing'' associated with
$G_1$ and $G_2$.

    Thus, if $g\in \Gamma_+$ and ${\tilde g}\in \Gamma_-$, with the above block decomposition,  a
computation shows that: 
\begin{equation}\label{e:SWpairing}
({\tilde g}, g)_{\Gamma_-,
\Gamma_+}^{SW}= \Det ({\tilde a} a {\tilde a}^{-1} a^{-1}) = \exp
(\text{trace} [\alpha, {\tilde \alpha}])\, ,
\end{equation}

\noindent where $g =
\exp (f)$, ${\tilde g} = \exp ({\tilde f})$, and $\alpha$ and
$\tilde \alpha$ are the $H_+ \to H_+$ blocks of $f$ and $\tilde f$
respectively. For details readers are referred to \cite{SW}.

\section*{Acknowledgements}
The authors want to thank the referee for valuable comments that have contributed to improve the paper.

\end{document}